\documentclass{amsart}
\usepackage{amsfonts,amssymb,verbatim,amsmath,amsthm,latexsym,textcomp,amscd, hyperref,comment}
\usepackage{epsfig}
\usepackage{amsmath,amsthm,graphics,float}
\usepackage{graphicx}
\usepackage{color}
\usepackage{url}

\PassOptionsToPackage{usenames,dvipsnames}{xcolor}
\usepackage{multirow,mathabx}
\usepackage{fancyhdr}
\usepackage[utf8]{inputenc}

\usepackage{mathtools}
\usepackage{amsthm,amssymb,latexsym,url,cancel}
\usepackage{pdfpages}

\usepackage[all]{xy}
\usepackage{subfig}
\usepackage{enumerate}
\makeatletter
\let\@@enum@org\@@enum@
\def\@@enum@[#1]{\@@enum@org[\normalfont #1]}
\makeatother

\newcommand\ba{\begin{align*}}
\newcommand\ea{\end{align*}}
\newcommand\be{\begin{enumerate}}
\newcommand\ee{\end{enumerate}}
\newcommand\bp{\begin{proof}}
\newcommand\ep{\end{proof}}
\newcommand\bpp{\begin{prop}}
\newcommand\epp{\end{prop}}
\newcommand\bpb{\begin{prob}}
\newcommand\epb{\end{prob}}
\newcommand\bd{\begin{defn}}
\newcommand\ed{\end{defn}}
\newcommand\bh{\begin{hint}}
\newcommand\eh{\end{hint}}

\newcommand\bC{\mathbb{C}}

\newcommand\bH{\mathbb{H}}
\renewcommand\Re{\mathbb{R}}
\newcommand\bR{\mathbb{R}}
\newcommand\bQ{\mathbb{Q}}
\newcommand\Q{\mathbb{Q}}
\newcommand\bZ{\mathbb{Z}}
\newcommand\Z{\mathbb{Z}}

\newcommand\FF{{\mathcal F}}

\newcommand\HH{{\mathcal H}}

\newcommand\MM{{\mathcal M}}

\newcommand\PP{{\mathcal P}}

\newcommand\SSS{{\mathcal S}}

\newcommand\TT{{\mathcal T}}

\newcommand\PMF{{\PP\kern-2pt\MM\FF}}

\newcommand\gam{\Gamma}

\DeclareMathOperator\harm{harm}

\newcommand{\Ga}{{\Gamma}}
\newcommand{\G}{{\Gamma}}

\newcommand{\pslc}{{\mathrm{PSL}_2 (\mathbb{C})}}
\newcommand{\pslr}{{\mathrm{PSL}_2 (\mathbb{R})}}
\newcommand{\pslz}{{\mathrm{PSL}_2 (\mathbb{Z})}}

\newcommand{\so}{{\mathrm{SO}}}
\newcommand{\su}{{\mathrm{SU}}}

\newcommand{\eps}{{\epsilon}}

\newcommand{\comme}{\operatorname{Comm}}

\renewcommand{\MR}[1]
{\href{http://www.ams.org/mathscinet-getitem?mr=#1}{MR#1}}

\usepackage{geometry}

\def\thetitle{Commensurators of thin normal subgroups and abelian quotients}
\def\theshorttitle{Commensurators of thin groups}
\usepackage{hyperref}
\hypersetup{
    pdftitle=   \thetitle,
   pdfauthor=  {T. Koberda, M. Mj}
}

\newtheorem{theorem}{Theorem}[section]
\newtheorem{lem}[theorem]{Lemma}
\newtheorem{lemma}[theorem]{Lemma}
\newtheorem{rmk}[theorem]{Remark}

\newtheorem{cor}[theorem]{Corollary}
\newtheorem{prop}[theorem]{Proposition}
\newtheorem{obs}[theorem]{Observation}

\newtheorem{qn}[theorem]{Question}
\newtheorem*{claim*}{Claim}
\newtheorem{claim}[theorem]{Claim}

\theoremstyle{remark}

\theoremstyle{definition}
\newtheorem{defn}[theorem]{Definition}

\begin{document}

\title[\theshorttitle]\thetitle
\date{\today}

% coauthor information
\author[T. Koberda]{Thomas Koberda}
\address{Department of Mathematics, University of Virginia, Charlottesville, VA 22904-4137, USA}
\email{thomas.koberda@gmail.com}
\urladdr{http://faculty.virginia.edu/Koberda}

\author{Mahan Mj}
\address{School of Mathematics, Tata Institute of Fundamental Research, 1 Homi Bhabha Road, Mumbai 400005, India}

\email{mahan@math.tifr.res.in}
\email{mahan.mj@gmail.com}
\urladdr{http://www.math.tifr.res.in/~mahan}

\thanks{The first author was partially supported by an Alfred P. Sloan Foundation Research Fellowship, by NSF Grant DMS-1711488,
and by NSF Grant DMS-2002596 while this research was carried out.
	The second author is  supported by  the Department of Atomic Energy, Government of India, under project no.12-R\&D-TFR-5.01-0500DST 
	and also in part by a Department of Science and Technology JC Bose Fellowship, and an endowment from the Infosys Foundation. Both authors were partially supported by the grant 346300 for IMPAN from the Simons Foundation and the matching 2015-2019 Polish MNiSW fund (code: BCSim-2019-s11).}

\date{\today}

\begin{abstract}
We give an affirmative answer to many cases of a question
 due to Shalom, which asks if the commensurator of a thin subgroup of a Lie group is discrete.
In this paper, let $K<\Ga<G$ be an infinite normal subgroup of an arithmetic lattice $\Ga$ in a rank one simple Lie group $G$,
such that the  quotient $Q=\Ga/K$ is infinite. We show that the commensurator of $K$ in $G$ is discrete, provided that $Q$ admits
a surjective homomorphism to $\bZ$.
In this case, we also show
that the commensurator of $K$ contains the normalizer of $K$ with finite index. We thus vastly generalize a result of the authors \cite{KM18},
which showed that many natural normal subgroups of $\pslz$ have discrete commensurator in $\pslr$.
\end{abstract}

\maketitle

\setcounter{tocdepth}{1}\tableofcontents

%%%%%%%%%%%%%%%%%%%%%%%%%%%
% START of body
%%%%%%%%%%%%%%%%%%%%%%%%%%%
%\setcounter{tocdepth}{0}\tableofcontents
%\mainmatter

\section{Introduction}\label{sec:intro}

Let $G$ be a semi-simple $\Q$--algebraic group, and let $G(\bZ)$ denote its group of integer points.
Roughly speaking, a subgroup $\gam$ of $G$ is called \emph{arithmetic}
if it is \emph{commensurable in a wide sense} with $G(\bZ)$~\cite{morris-book}.
That is, there is an element $g\in G$ such that the group $G(\bZ)\cap\gam^g$ has
finite index in both $G(\bZ)$ and $\gam^g$.
In general, if $G$ is an algebraic group and $\gam<G$ is a subgroup, we write
$\comme_G(\gam)$ for the \emph{commensurator} of $\gam$ in $G$, i.e.~the subgroup consisting of $g\in G$
such that $\gam\cap\gam^g$ has
finite index in both $\gam$ and $\gam^g$.  The Commensurability Criterion for
Arithmeticity due to Margulis~\cite{margulis,morris-book} characterizes arithmetic subgroups
of algebraic groups via their commensurators.
A convention we shall follow throughout in this article: whenever we refer to
a semi-simple Lie group, we shall mean  a connected semi-simple real Lie group with no compact factors, unless noted otherwise.

\begin{theorem}[Margulis]
	Let $G$ be a semi-simple Lie group with no compact factors and let $\gam$ be an irreducible lattice in $G$.
	Then $\gam$ is arithmetic
	if and only if $\comme_G(\gam)$ is dense in $G$.
\end{theorem}

Here, we are primarily concerned with the discreteness properties of commensurators of \emph{thin groups}, a class of groups which has received a large amount of attention in recent
years~ \cite{sarnak-thin}.
A subgroup $K<G$ is \emph{thin} if $K$ is discrete and
Zariski dense in $G$, and if $G/K$ has infinite volume with respect to the Haar measure on $G$. Thus, $K$ fails to be a lattice in $G$
only by virtue of having infinite covolume in $G$. Natural examples of thin groups arise from infinite index
Zariski dense subgroups of lattices in $G$.

In the present manuscript, we continue our previous investigations from \cite{KM18}  of the following question due to Shalom (see especially
\cite{shalom-gafa}, wherein the problem has its genesis):

\begin{qn} \cite{llr}
	\label{mainq}
	Let
	$K$
	be a thin subgroup of a semi-simple Lie group
	$G$.  
	\begin{enumerate}
\item 	Is
	the commensurator $\comme_G(K)$ of
	$K$ in $G$ discrete?
	\item In particular, is the normalizer of $K$ in $G$ of finite index in $\comme_G(K)$?
	\end{enumerate}  
\end{qn}

For an infinite normal subgroup $K$ of a lattice $\Ga$, the two sub-questions of Question \ref{mainq} are equivalent.
Indeed, the commensurator
of  $K$ contains its normalizer, which contains $\Ga$. Since $\Ga$ is a lattice, we see that
if $\comme_G(K)$ is  discrete then it is a finite index super-lattice
of $\Ga$.  For the other implication, any such $K$ is discrete and Zariski  dense, and thus has a discrete normalizer
(cf.~Lemma~\ref{lem:normalizer}). Since the normalizer of $K$ contains $\gam$ and since $\gam$ has finite covolume, we have
that the normalizer of $K$ is itself a lattice. Thus, if the commensurator of $K$ contains the normalizer of $K$ with finite index then the
commensurator is discrete.

Positive answers to Question \ref{mainq} are known for all finitely generated thin subgroups $K$ of $\pslr$ and $\pslc$
\cite{greenberg,llr,mahan-commens}, and for
thin subgroups of a semi-simple Lie groups with limit set
a proper subset of the Furstenberg boundary \cite{mahan-commens}. Here, the \emph{limit set} is a generalization of the limit set occurring
in the theory of Kleinian groups, and is a minimal nonempty
 closed invariant subset of the Furstenberg boundary for a group acting on the corresponding
symmetric space (see~\cite{benoistgafa}).

We were thus prompted in \cite{KM18} to address
Question \ref{mainq} when the ambient Lie group is the simplest possible, viz.~$\pslr$, for thin groups whose limit sets consist of the
entire Furstenberg boundary, i.e.~ $S^1=\partial\bH^2$. More generally, natural examples of thin groups with limit set equal to the
Furstenberg boundary come from normal subgroups of rank one lattices. This general problem provides the context for this paper.

\subsection{Main result}
Since many rank one arithmetic lattices surject onto nonabelian free groups, every finitely generated group can be realized as a quotient of an
arithmetic lattice. Observe in particular, that all finitely generated free groups arise as finite index subgroups of $\Ga(2)$, the level-two
congruence subgroup of $\pslz$,  and therefore  all infinite, finitely generated groups arise as quotients of a rank one  arithmetic lattice by
a thin normal  subgroup. This level of generality has led us to impose 
some natural algebraic conditions on the quotient $Q$. 
We will establish the following result, which handles normal subgroups with ``nice" quotients.

\begin{theorem}\label{thm:main-general}
Let $\Ga<G$ be an arithmetic lattice in a rank one simple Lie group $G$ and let $K<\Ga$ be an infinite
normal subgroup. Write $Q=\Ga/K$ for
the corresponding quotient group. Then the group $\comme_G(K)$ is discrete, provided that
the group $Q$ admits a surjective homomorphism to $\bZ$.
Under  these hypotheses, the commensurator of  $K$ in $G$ contains the normalizer of $K$ with finite index.
\end{theorem}

The reader is directed to Theorem~\ref{thm:q-b1}
for the context and proof surrounding the  main result.
Note that the hypotheses of  Theorem \ref{thm:main-general} are never satisfied for irreducible lattices in
higher rank 
%(due to Margulis' normal subgroup theorem~\cite{margulis,morris-book}),
 nor for lattices
in  the rank one simple Lie groups $\mathrm{Sp}(n,1)$ for $n\geq 2$, nor in
the exceptional group $F_{4}^{-20}$. 
This is because lattices in these Lie groups have Kazhdan's Property (T). Thus,
Theorem \ref{thm:main-general} is vacuously true in these cases.
%; see Proposition~\ref{prop:rep} below.
Therefore
in the course of establishing Theorem~\ref{thm:main-general},
we pay  exclusive attention to $G\in\{\so(n,1),\su(n,1)\}_{n\geq 2}$, which give rise to real and complex hyperbolic spaces as the
associated symmetric spaces of  noncompact type, respectively.

In \cite{KM18} we answered Question \ref{mainq} in the special case that $K$  is the commutator subgroup of
$\Ga$, where $\Ga< \pslz$ is a finite index normal
subgroup of $\pslz$ contained in a principal congruence subgroup $\G(k)$ for some $k\geq 2$. We vastly generalize this result, since 
if $K=[\Ga,\Ga]$ has infinite index in $\Ga$ then $K$ falls under the purview of Theorem~\ref{thm:main-general}.

\subsection{Tools and Techniques}
The main theorems and techniques of \cite{KM18} are the starting point
of this paper.

\noindent {\bf Preserving lines with holes:}
An important technical tool introduced in \cite{KM18} was that of a homology pseudo-action. We adapt it here to the notion of
\emph{preservation of lines with holes}.
Let $\Ga$ be a lattice in  a rank one simple Lie group $G$, let $K<\Ga$ be a normal subgroup, and let $Q=\Ga/K$. Quite generally,
for $g \in G$ we say that $g$ preserves  $Q$--lines with holes if for all $\gamma \in \Ga$,
there exists $N > 0$ such that 
\[\gamma^n \cong (\gamma^n)^g \, (\mathrm{mod} \, K)\textrm{ for all $n\in N\bZ$}.\]
The terminology arises from thinking of infinite cyclic groups as ``lines" and a finite index subgroup of an infinite cyclic group as a ``line with holes". We direct the
reader to
Section \ref{sec:pseudoaction} for a detailed discussion. 

The usefulness of preserving lines with holes is illustrated by the following purely group-theoretic fact,
 which provides a rather general criterion for deciding non-commensurability (see Theorem \ref{thm:comm-nonab}):

\begin{theorem}\label{maincrit}
	Let $\Ga<G$, let $K<\Ga$ be normal, and let $Q = \Ga/K$. If \[g \in \comme_GK \cap \comme_G\Ga,\] then 
 $K^g:=g^{-1}Kg$ preserves $Q$--lines with holes.
\end{theorem}

\noindent {\bf Harmonic forms and maps:}
The other principal tool used in this paper comes from harmonic forms 
and harmonic maps via Hodge theory.
These include classical
Hodge theory and its $L^2$ analogue for non-compact manifolds.
 Preservation of $Q$--lines with holes, or equivalently, lines with holes
in $\Ga$ modulo the normal subgroup $K$ can be promoted to something stronger: the harmonic 
form
 allows us to convert the ``coarse'' lines in $\Ga/K$ into actual maps to
$\bR$, i.e.\ it allows us to ``fill the holes" of coarse lines in a canonical fashion, and thus find canonical $G$--invariant maps
to $\bR$.\\

\noindent {\bf Discrete patterns:}
Harmonic maps are coupled with the notion of discrete patterns, an idea going back to Schwartz \cite{schwartz-inv}, and
which was exploited in proving discreteness of commensurators in \cite{llr,mahan-commens}.
Throughout the paper, many of our ideas and methods are inspired by the basic example of arithmetic 
hyperbolic surfaces as well  as the special case $K=[\Ga,\Ga]$, and in some places we explicate the underlying geometric intuition.
In the context of $\pslr$ and hyperbolic surfaces, Teichm\"uller--theoretic notions such as
zeros and saddle connections of abelian differentials provide us the necessary discrete patterns that are preserved by the commensurator
when the underlying surface has positive genus {\it and} lines with holes in the integral homology are preserved.
Preservation of such discrete patterns finally ensures that the commensurator is discrete.
With the notion of preserving homological lines with holes in place,  the discussion for  lattices in $SO(n,1)$ and $SU(n,1)$ splits into 
uniform and non-uniform cases. For uniform lattices, we use Hodge theory coupled with a Lie--theoretic idea that we learned from
 Venkataramana and Agol \cite{venky,agol-form}.
For non-uniform lattices, we use $L^2-$Hodge theory along with the fact that
preservation of homology lines with holes 
 guarantees the preservation of a discrete pattern given by horoballs.
Discreteness of a pattern-preserving subgroup is an essential ingredient in the non-vanishing cuspidal cases:
see the proof of Theorem~\ref{thm:q-b1}, especially Claim~\ref{claim1} therein.\\

 \noindent {\bf Relationship with existing literature:} The previous works \cite{llr,mahan-commens} on discreteness of commensurators
 derived discreteness by showing that the commensurator preserves a ``discrete geometric sub-object" or   ``pattern" in the sense of
 Schwartz \cite{schwartz-pihes}. These may be regarded as a collection of geometrically defined subspaces of the domain symmetric
 space $X$. We refer the reader to Appendix \ref{sec-app} for the material on patterns that will be used in this paper. There is a shift in focus in this paper, as we look at naturally defined dual objects. The canonical
 nature of harmonic maps ensures that they are preserved by the commensurator.
We  derive  much  of our inspiration from Shalom's  work
  \cite{shalom-annals,shalom2000,shalom-gafa}.
 
 \subsection{Structure of the paper}

Section~\ref{sec:commensurator} contains an account of the general tools from the theory of lattices in Lie groups which we will need.
Section~\ref{sec:pseudoaction} describes preservation of lines with holes in detail.
Section~\ref{sec:hodge} introduces the notion of a discrete invariant set as it arises from classical and $L^2-$Hodge theory. In the same
section, the commensurator of a form is introduced and the construction of an invariant harmonic form
is carried out. Section~\ref{sec-mainthm} proves
 Theorem~\ref{thm:main-general}.

\bigskip
\noindent {\bf Remarks on notation:}
Throughout this paper, we will use the notation $K$ to denote a subgroup a discrete group. Usually, this will be a normal subgroup
of an arithmetic lattice $\Ga$. In particular, $K$ will generally not denote a maximal compact subgroup of the ambient Lie group $G$.
We will use $N$ to denote a positive integer, as opposed to the more common notation
of the unipotent subgroup in the Iwasawa decomposition of a 
 semi-simple Lie group. The Iwasawa decomposition will  be used briefly
in the proof  of  Claim~\ref{claim1}, but  no  confusion will arise.
We will use the exponentiation shorthand for conjugation
in groups, so that $K^g=g^{-1}Kg$, where  $K$ and $g$ are contained in an ambient group.
The group $G$ will denote an ambient Lie group, which will typically be 
$\{\so(n,1),\su(n,1)\}_{n\geq 2}$
unless otherwise explicitly noted.

\section{Generalities on discrete subgroups of Lie groups}\label{sec:generalities}\label{sec:commensurator}

In this section, we gather some general facts about Zariski dense discrete subgroups of semi-simple Lie groups which we will require
in this article.

\subsection{Zariski dense subgroups and commensurators}
We begin with the following general fact about normalizers of discrete groups. The statement and proof are
 contained as Lemma 2.1 in~\cite{KM18}, and so we omit the  proof.

\begin{lem}\label{lem:normalizer}
	Let $G$ be a simple Lie group and let $\gam\, < \, G$ be a discrete Zariski dense subgroup. Then the normalizer $N_G(\gam)$ is again
	discrete.
\end{lem}

The following well--known fact will be used throughout the paper.
\begin{lem}
Let $G$ be a simple real group and let $H<G$ be a Zariski dense subgroup. If $H$ is not discrete then $H$ is dense.
\end{lem}

Indeed, since $H$ is not discrete then the topological closure $\overline{H}$ of
$H$ has the property that the component $H^0$ of $\overline{H}$ containing the
identity is a Zariski dense subgroup of $G$ which has positive
dimension, and therefore must be all of $G$; indeed the tangent space to $H^0$ at the identity coincides with the tangent space to $G$,
and so $H^0$ contains a neighborhood of the identity in $G$, which generates the identity component of $G$.
We remark that if $G$ is allowed to be a complex group then one must assume that $H$
is not precompact, as can be seen from the Zariski density of
the unit complex numbers in $\bC$ for instance.

The following Lemma generalizes the corresponding statement in~\cite{KM18} for $\pslr$.

\begin{lemma}\label{lem:power}
	Let $\Ga_0$ be a lattice in a noncompact simple Lie group $G$.
	Let $\gam$ be a subgroup of $G$ containing $\gam_0$ such that   there exists an $N>0$ satisfying the property that
	for all $g\in\gam$, we have $g^N\in \Ga_0$. Then $\gam$ is also discrete.
\end{lemma}

\begin{proof} 
	We have that
	$G$ acts  by isometries on an associated symmetric space $X$ of noncompact type.
	Since $\Ga_0$ is a lattice, there exists $\eps>0$ such that any semi-simple element of $\Ga_0$ has translation length
	on $X$ at least
	$\eps$.
	Since $G$ is simple and $\gam$ is Zariski dense, it follows that $\gam$ is either discrete
	or dense in $G$. We argue by contradiction. If $\gam$ is dense, then since the property of being
	semi-simple is an open condition and since translation lengths of semi-simple elements of $G$ coincide with $\bR_{>0}$,
	there exists a semi-simple element $g \in \gam$ such that the translation length of $g$ is less than $\eps/2N$.
	Hence $g^N$ is a semi-simple element with translation length at most $\eps/2$. In particular,
	$g^N \notin \Ga_0$, which yields a contradiction.
\end{proof}

We remark that Lemma~\ref{lem:power} is false for merely discrete subsets of $G$, since even the square roots of a fixed matrix can
fail to be a discrete set. If $G$ has rank one then one can allow $\Ga_0$ to be a more general subset of $G$.

Let $G$ be a 
semi-simple Lie group and let $\gam<G$ be a subgroup. As usual, we write $\comme_G(\gam)$ to denote
its commensurator in $G$.
We shall need the following special case of a general theorem of Borel \cite[Theorem 2]{borel} (see  \cite[p. 123]{zimmer-book}). This will
be the only real use of arithmeticity of the ambient lattice $\Ga$ in Theorem~\ref{thm:main-general}.
Strictly speaking, the statement of Proposition 6.2.2 in~\cite{zimmer-book} is for the full group of integral points in an ambient group.
The reader will note
however that the only salient feature of the group of integral points which is used is its Zariski density. Thus, we obtain the following conclusion:

\begin{prop}\label{commens=commesn}
	Let $\Ga < G$ be an arithmetic lattice in a  semi-simple algebraic $\bQ$--group and let $K < \Ga$ be
	a Zariski dense subgroup. Then $\comme_G(K)<\comme_G(\Ga)$.
	Suppose furthermore that the center of $G$ is trivial. Then
	$\comme_G(\gam)$ coincides with the $\bQ$--points of $G$.
\end{prop}

The hypothesis that $G$ has trivial center in the second part of Proposition~\ref{commens=commesn} is crucial. For instance,
the commensurator of $\mathrm{SL}_2(\bZ)$ properly contains $\mathrm{SL}_2(\bQ)$.
The reader will observe that
throughout this paper, we will implicitly assume that $K$ is a Zariski dense subgroup of an arithmetic lattice, though in the statement
of Theorem~\ref{thm:main-general}, we only assume that $K$ is infinite and normal.
This latter assumption implies that $K$ is indeed Zariski dense:

\begin{prop}\label{prop:inf-zd}
Let $K<\Ga$ be an infinite normal subgroup of an irreducible lattice in a 
semi-simple algebraic group $G$. Then $K$ is Zariski dense in $G$.
\end{prop}
\begin{proof}
Let $\Lambda$ denote the limit set of $K$. Since $K$ is infinite, $\Lambda\neq\varnothing$, since the limit set  consists of the  limit
points of $K$  in the Furstenberg boundary of $G$. Let  $p\in\Lambda$. If $\gamma\in\Ga$  then
$\gamma(p)\in\Lambda$, since  $K$ is normal in $\Ga$. It follows that $\Lambda$ is
a   closed, non-empty $\Ga$--invariant subset of  the Furstenberg boundary.  It therefore  contains all of the  limit set  of  $\Ga$
by the lemma in Section 3.6 of~\cite{benoistgafa}. It  follows that $\Lambda$ is equal  to the limit set of $\Ga$.

Since $\Ga$ is Zariski dense, so is $K$. Else, if $K$  were contained in a proper Lie subgroup $H<G$ then $\Lambda$ would be contained
in the Furstenberg boundary of $H$, which in turn
 is not Zariski dense in the 
Furstenberg boundary of $G$. However the limit set of  $\Ga$ is Zariski  dense  in the boundary: see the remarks  at  the beginning  of
Section 3 of~\cite{benoistgafa}, especially  the lemma in Section 3.6
of that paper. This is a contradiction.
\end{proof}

The following technical fact will be used several times in this paper, and we extract it for modularity.

\begin{lem}\label{lem:comm-discrete0}
Let $K<G$ be a Zariski dense subgroup of a simple algebraic group $G$,
and let \[K^G=\langle\{ K^g \mid g\in\comme_G(K)\}\rangle\]
be the subgroup of $G$ generated by the conjugates of $K$ by $g\in\comme_G(K)$. If $K^G$ is a
discrete subgroup of $G$, then $\comme_G(K)$ is discrete.
\end{lem}

It is a trivial though useful observation that $K^G<\comme_G(K)$.

\begin{proof}[Proof of Lemma~\ref{lem:comm-discrete0}]
We have immediately that $K<\comme_G(K)$, since $K$ normalizes itself. We therefore conclude that
$\comme_G(K)$ is  Zariski dense and  hence is either discrete or dense in $G$.
If $\comme_G(K)$ is dense then there is a sequence $g_i\to 1$ of
nontrivial group elements in
$\comme_G(K)$ converging to the identity. We write $K_i=K^{g_i}$, and we observe that $K_i<K^G$ for each $i$. Choosing finitely
many elements $\{k_1,\ldots,k_m\}\subset K$ which generate a Zariski dense subgroup $K_0<G$,
we have that if $g_i$ is nontrivial then it
cannot fix the entire collection $\{k_1,\ldots,k_m\}$,  since then  $g_i$  would  centralize $K_0$, contradicting Zariski density of $K_0$ and
the simplicity of $G$. However, as $i$ tends to infinity, the conjugation action of $g_i$ on $\{k_1,\ldots,k_m\}$
tends to the identity. Thus, viewing $G$ as a matrix group, we have that
 $\{k_1^{g_i},\ldots,k_m^{g_i}\}$ converges to $\{k_1,\ldots,k_m\}$ in any matrix norm. Since $K_0^{g_i}<K_i<K^G$,  the last of
  which is
 discrete, we have that
 $\{k_1^{g_i},\ldots,k_m^{g_i}\}=\{k_1,\ldots,k_m\}$ element--wise for $i\gg 0$, and hence that $g_i$ commutes with $K_0$  for
 $i\gg 0$. Again using the fact that $K_0$ is Zariski dense and
 $G$ is simple and hence  center--free, we  conclude that $g_i$ is the identity for $i\gg 0$. This is
 a contradiction, whence it follows that $\comme_G(K)$ is discrete.
\end{proof}

The argument in Lemma~\ref{lem:comm-discrete0} even shows that only the set \[\bigcup_{g\in\comme_G(K)} K_0^g\]
need be discrete in order
to conclude the discreteness of $\comme_G(K)$, for  an arbitrary Zariski  dense  subgroup $K_0<K$.

\section{Preservation of lines with holes}\label{sec:pseudoaction}

In this section, we develop some ideas which originate in homological algebra
and which play a central role in this paper, with the goal of producing a criterion for showing
that a particular group element does not commensurate a given subgroup. The historical motivation comes from Chevalley--Weil
theory~\cite{cw,gaschutz},
and which was developed in ~\cite{KM18} under the name of a \emph{pseudo-action}.

Throughout this section,
let $\gam<G$, let $K<\gam$ be a normal subgroup, and let $g\in \comme_G(\Ga)$.
We write $Q=\gam/K$ for the quotient group.
Conjugating by $g\in G$, we obtain groups $K^g<\gam^g$ and a corresponding quotient $Q^g:=\gam^g/K^g$.

For $\gamma \in \Gamma$, we shall refer to the cyclic group $\langle \gamma \rangle$ as a \emph{ $\gamma-$line in $\Gamma$}. Further, any finite index subgroup $\langle \gamma^N \rangle$ of $\langle\gamma\rangle$
	(considered for arbitrary $\gamma \in \Gamma$ and a positive integer $N$) will be referred to as a 
	\emph{ $\Gamma-$line with holes.}  For any $\gamma \in \Gamma$ and $g\in \comme_G(\Ga)$,
	there exists a positive integer $N$, such that $(\gamma^g)^N \in \Gamma$. Hence, for any $\gamma \in \Gamma$, and
	$ g\in \comme_G(\Ga)$, (the conjugation action by) $g$ sends some  $\gamma-$line with holes to a  $\Gamma-$line with holes.

\begin{defn}\label{def-pa}
	The element $g\in \comme_G(\Ga)$ \emph{preserves $Q$--lines with holes} if for all $\gamma \in \Gamma$
	there exists an integer $N>0$ such that
	\[\gamma^n \equiv (\gamma^n)^g  \pmod  K\] for all $n\in N\bZ$. That is, there exists $N > 0$ such that $x_m=[\gamma^{mN},g] \in K$
	for all $m\in\bZ$.
\end{defn}

Thus if $\gamma^N$ and $(\gamma^N)^g$ should both be elements of $\Ga$ (which they are after passing to
multiples of a sufficiently large $N$ since  $g$ commensurates $\Ga$),
then  one can compare their images in $Q=\Ga/K$. If $g$ preserves $Q$--lines with holes
then they must represent the same element of $Q$. 
A special case of Definition \ref{def-pa} is given by the following:

\begin{defn}\label{def-trivialaction-comm-gam} In Definition \ref{def-pa}, if
	we specialize to the case where	$K$ is the commutator subgroup $ [\Ga,\Ga]$ (so that in particular
	$Q=	H_1(\gam,\bZ)$), we say that $g$ \emph{preserves homological lines with holes in $\Gamma$}.
\end{defn}

%\begin{comment}
%	content...
	
%	Observe that since $g$ is assumed to commensurate $\gam$, there is always an integer $N$ such that $(z^g)^N\in\gam$, and so it makes
%	sense to consider the homology class of this element. Observe also that preserving homological lines with holes is nothing other than
%	preservation of $Q$--lines with holes in the case $Q=H_1(\gam,\bZ)$.
%\end{comment}

The usefulness of preservation of homological lines with holes will become apparent when one considers its cohomological consequences in Section~\ref{ss:pa2}.
For now, consider the set of all elements $g \in \comme_G(\Ga)$ that preserve $Q$--lines with holes.
It is not  difficult to see  that this subset of $G$ is actually a monoid.
Clearly the  identity lies in this set. Moreover, if $g,h$ preserve $Q$--lines with holes then for all
$\gamma\in\Ga$, there is an $N=N(g,\gamma)$ such  that 
$[\gamma^{N},g] \in K$. Then, $(\gamma^N)^g=\gamma^N\cdot k\in\Gamma$,
so  there is an $M=M(h,\gamma^N\cdot k)$ such that
$[(\gamma^{N}\cdot k)^M,h] \in K$. This shows that \[\gamma^{NM} \equiv (\gamma^{NM})^{gh}  \pmod  K,\] which implies that
the set of elements of $\comme_G(\Ga)$ which preserve $Q$--lines with holes does in fact  form  a monoid.
It is not clear that
inversion
of elements is possible within this set, however. 
We will not require this  monoidal structure in the sequel, though we abstract out the following fact:

\begin{comment}
	content...
 do note
that if  $K_1,K_2<\comme_G(\Ga)$ are subgroups  which preserve $Q$--lines with holes then the subgroup
$\langle K_1,K_2\rangle<\comme_G(\Ga)$
also preserves $Q$--lines with holes.
\end{comment}

\begin{obs}\label{obs:monoid}
	Consider the set $C\subset\comme_G(\gam)$ consisting of elements which preserve $Q$--lines with holes. Then $C$ is closed under
	multiplication of group  elements and contains the identity, and is therefore a monoid. In  particular, if $K_1,K_2\subset C$ are subgroups,
	then the group  \[\langle K_1,K_2\rangle<\comme_G(\gam)\] is contained in $C$.
\end{obs}

The following is the basic result about preservation of $Q$--lines with holes.

\begin{theorem}\label{thm:comm-nonab}
	Let $\Ga < G$, let $K$ be a normal subgroup of $\Ga$, and let $Q = \Ga/K$.
	Suppose that \[g \in \comme_G \Ga \cap \comme_G K.\] Then
	$K^g$ preserves $Q$--lines with holes.
\end{theorem}

\begin{proof}
	Let $z\in K^g$ and let $\gamma\in\gam$ be
	arbitrary fixed elements. For $N\gg 0$ we have that $\gamma^N\in \gam\cap\gam^g$ and $(\gamma^N)^z\in\gam$.
	Let $a = (\gamma^N)^z$ and $b = \gamma^N$. We have that $a^m , b^m  \in \Ga$ for all $m\in \Z$.
	
	Since $z\in K^g$ and since $K^g$ is
	normal in $\gam^g$, we have that \[a \equiv b \pmod{K^g}.\] Hence, \[a^m \equiv b^m  \pmod{K^g}\] for all $m\in \Z$.
	Thus, the commutators $$x_m := [\gamma^{mN}, z] = a^mb^{-m}$$ have the property that $x_m \in K^g$ for all $m \in \Z$. 
	It is also clear that $x_m\in\gam$ for all $m \in \Z$.

	Since $K$ and $K^g$ are commensurable,
	the collection of elements \[\{x_m=a^mb^{-m}\}_{m\in\bZ}\] has the property that for some
	$s \neq t$, the elements $x_s=a^s b^{-s}$ and $x_t=a^t b^{-t}$ lie in the same
	right coset of $K\cap K^g$ in $K^g$, as follows immediately from the pigeonhole principle.
	
	It follows that there exists an element $k \in K$  such that
	\[k a^s b^{-s} = a^t b^{-t}.\]
	Therefore, we see that
	\[a^{-t}ka^s = b^{s-t},\] which furnishes an element
	$k' \in K$
	such that
	$k'a^{s-t} = b^{s-t}$.
	
	Thus, there exists $M= s-t \neq  0$  such that $a^M \equiv b^M\pmod K$. In particular, $z$ preserves $Q$--lines with holes, the desired
	conclusion.
\end{proof}

In the sequel, we will be interested in specific cases in which $Q$--lines with holes are preserved, and especially the case where
$Q$ is the integral homology of $\Ga/K$.

We now discuss a mild generalization of the notion of preserving homological lines with holes (Definition~\ref{def-trivialaction-comm-gam}). 
%We retain the notation $S=X/\Ga$, $S^g$, $W$ etc.~from Section \ref{ss:pa2}.
Let $Q=\Ga/K$ be
a quotient group.
Clearly, $H_1(Q,\bZ)$ is a quotient of $H_1(\Ga,\bZ)$. 

Let $\gamma\in \Ga$ and let
$g\in\comme_G(\Ga)$.  There  is an integer $N>0$ such that  $\{\gamma^n,(\gamma^g)^n\}\subset\Ga$
for all $n\in N\bZ$. We  can then
compare the homology classes of  $\gamma^n$ and $(\gamma^g)^n$  in $H_1(\Ga,\bZ)$, and  hence in $H_1(Q,\bZ)$. As before,
we say that $g$ \emph{preserves homological lines with holes} in $Q$ if for all
$\gamma\in\Ga$, there exists an integer $N > 0$ such that 
for all $n\in N\bZ$, the homology classes of $\gamma^n$ and $(\gamma^g)^n$ in $H_1(Q,\bZ)$ are equal. 

Let $Q^{ab}$ denote the abelianization of $Q$. Then the condition that $g$ preserves homological lines with holes in $Q$ is equivalent to saying that $g$ preserves $Q^{ab}-$lines with holes in the sense of Definition~\ref{def-pa}.

When $b_1(Q)>0$ then Theorem~\ref{thm:comm-nonab} above furnishes the following commensurability criterion, whose proof
is straightforward now.

\begin{theorem}\label{thm:comm-nonab-b1}
	Let $Q=\gam/K$, let \[g \in \comme_G \Ga \cap \comme_G K,\]. Then $K^g$ preserves homological lines with holes in $Q$.
\end{theorem}

\begin{proof}
Let $Q_0$ be a quotient of $Q$, and let $h \in K^g$. Since $h$ preserves $Q-$lines with holes by Theorem~\ref{thm:comm-nonab},  it also preserves $Q_0-$lines with holes. Specializing to $Q_0=Q^{ab}$ proves the result.
\end{proof}

In particular, when the commensurator of $\gam$ in $G$ contains the commensurator of $K$,
we have that \[K^G=\langle K^g\mid g\in\comme_G(K)\rangle\] preserves homological lines with holes in $Q$.
We remark that in our applications,
$\comme_GK<\comme_G\Ga$ by Proposition~\ref{commens=commesn}.

\section{Homological lines with holes and Hodge theory}\label{sec:hodge}
The goal of this section is to translate between preservation of lines with holes and the existence of commensuration-invariant
%geometric objects. The geometric objects we consider are 
harmonic 1-forms. We shall first deduce cohomological consequences of preserving
homological lines with holes.
%and discrete patterns. 
%These  will provide us with the essential tools to conclude discreteness of  commensurators.

\subsection{Preserving homological lines with holes and cohomological consequences}\label{ss:pa2}

For the purposes of this subsection, let $G$ denote a semisimple Lie group with no compact factors,
with associated symmetric space of nonpositive
curvature $X$.
Let $\Ga$ be a lattice in $G$ and let $g \in \comme_G (\Ga)$. We write $S = X/\Ga$ and $S^g = X/\Ga^g$.
Since $g \in \comme_G (\Ga)$,  the group $\Ga \cap \Ga^g$ is of finite index in both $\Ga$ and $\Ga^g$.
Let $W= X/(\Ga\cap \Ga^g)$ denote the corresponding common cover of
$S$ and $S^g$.
We shall refer to
$S$ and $S^g$ as \emph{conjugate manifolds}  and $W$ as their \emph{minimal common cover}. Here, $W$ depends on $g$. However, since
 $g$ will be fixed throughout, we will suppress it from the notation.
We will also fix a differential $1$--form $\omega$ on $S$.
Let $p:X \longrightarrow S$ denote the universal covering map. Note that the form $1$--form
$p^\ast \omega$ is a 1-form on $X$. In applications in the sequel, $\omega$ will be a harmonic
form.

The element $g\in G$ is an isometry of $X$ and hence acts on differential forms on $X$ via pullback. The form 
$g^\ast p^\ast \omega$ is a 1-form on $X$ which is invariant under $\Ga^g$
and hence descends to $S^g$. The resulting 1-form on the quotient manifold $S^g$ is denoted
by $\omega^g$.
Let $q: W \longrightarrow S$ and $q^g: W \longrightarrow S^g$ denote the natural covering maps. Denote $q^\ast \omega$ by $\omega_W$
and $(q^g)^\ast \omega^g$ by $\omega_W^g$.

We shall also need to set up notation for $g-$conjugates of cycles and loops, as basepoints will play an important role in what follows.
Let $o \in W$ be a basepoint. By choosing a
lift $\tilde{o} \in X$ and by joining $\tilde{o}$ to $g.\tilde{o}$ by a geodesic segment in $X$ and projecting back to $W$, we 
obtain  a natural geodesic segment $[o,g.o]$ in $W$, where  $g.o$ denotes the image of $g.\tilde{o}$ under the covering projection.
Thus, $g.o$ may be regarded as a new basepoint for integrating chains against a pulled back form.

Now suppose that $\alpha$ is a loop in $W$ representing an element $h \in \pi_1(W)$ such that $h^g$ also
belongs to $\pi_1(W)$  (where here we have identified $\pi_1(W)$ with $\Ga \cap \Ga^g$). 
Lifting $\alpha$ to a path $\tilde{\alpha}$ in $X$,
translating by $g$ and quotienting $X$ by $\Ga \cap \Ga^g$ we obtain a new loop denoted $g.\alpha$ on $W$ based at $g.o$.
Here, we use notation that is similar to the case of a genuine $g-$action on $W$, though the action is well-defined only on the universal
cover $X$. 

The concatenation $[o,g.o] \ast g.\alpha \ast \overline{[o,g.o]}$ gives a loop based at $o$, where
$\overline{[o,g.o]}$ denotes $[o,g.o]$ parametrized in the opposite direction from $g.o$ to $o$.
We denote this loop as $\alpha^g$: $$\alpha^g = [o,g.o] \ast g.\alpha \ast \overline{[o,g.o]}.$$
Finally, for $\sigma$   any  closed, oriented loop on $W$, based at $o$ say, the $n^{th}$ power of the loop $\sigma$ will be the loop
which traverses the loop $\sigma$ a total of $n$ times. The result will be  denoted by $\sigma^n$.

\begin{rmk}\label{rmk-caveat}
	A subtlety in the following lemma needs to be noted.
	On the one hand, the hypothesis is about preserving homological lines with holes in $\Gamma$.
	The conclusion, on the other hand, is about cohomology classes in the common minimal cover $W$.
	The reason for this is that the pullback of $\omega$ to $X$ and its pullback by $g$ are both invariant under
	$\gam\cap\gam^g$, though
	not necessarily by $\gam$ nor $\gam^g$. Thus, $\omega^g_W$ is well-defined as a form on $W$,
	but does not necessarily live in $S$.
\end{rmk}

\begin{lem}\label{trvialcohomolaction}
	Let \[\{\Ga, S, g, S^g, W, \omega_W, \omega_W^g\}\] be as above,
	Suppose that $g$ preserves homological lines with holes in $\gam$. Then
	we have $[\omega_W] = [\omega_W^g]$ as  elements of $H^1(W,\bR)$.
\end{lem}

The importance of Lemma~\ref{trvialcohomolaction} will become apparent in Section \ref{sec:patterns}, particularly Corollary~\ref{cor:hodge-paction}. It follows from the Hodge theorem that  if $\omega_W$ is a harmonic form representing $[\omega_W]\in H^1(W,\bR)$, then $\omega_W=\omega_W^g$ as forms, and not just as cohomology classes.

\begin{proof}[Proof of Lemma~\ref{trvialcohomolaction}]
	We continue with the notation from the discussion before the statement of the lemma.
	Let $\sigma$  be any  closed loop on $W$ based at $o$. Since $g$ commensurates $\gam$, we may choose
	$n>0$ such that  $\sigma^n$ and $(\sigma^n)^g$ are both cycles, and so are viewed as loops based at $o$.
	Observe that if
	$h$ denotes the element of $\pi_1(W,o)$ represented by $\sigma^n$ then the loop $(\sigma^n)^g$ represents
	the group element $h^g\in \pi_1(W,o)$.
	
	Since $g$ is assumed to preserve homological lines with holes in $\gam$, there exists an integer
	$N> 0$ such that $\sigma^N$ and  $ (\sigma^N)^g$  represent the same element of $H_1(S, \bZ)$. Indeed, for any differential
	$1$--form $\omega$ on $S$, we have
	
	\begin{equation}\label{eq1}
		\int_{q(\sigma^N)} \omega = \int_{q((\sigma^N)^g)} \omega=\int_{q(g.\sigma^N)}\omega,
	\end{equation}
	where $q:W \longrightarrow S$ is the covering projection, and where the second inequality holds
	because the integrals of $\omega$ along $[o,g.o]$ and $\overline{[o,g.o]}$ cancel each other.
	Note that the integrals in Equation \eqref{eq1} are over $S$.
	
	Next, by the definition of the pullback form $\omega_W=q^\ast \omega$,
	we have that \[\int_{\sigma^N} \omega_W = \int_{q(\sigma^N)} \omega, 
	\textrm{  and  } \int_{(\sigma^N)^g}\omega_W = \int_{q((\sigma^N)^g)} \omega.\] 
	
	Combining the equations above, we obtain
	
	\begin{equation}\label{eq2}
		\int_{\sigma^N} \omega_W  = \int_{(\sigma^N)^g}\omega_W=\int_{g.\sigma^N}\omega_W,
	\end{equation}
	where all the integrals in Equation \eqref{eq2} are over $W$.

	Finally, we observe
	that by the definition of the pullback $\omega_W^g$, we have
	\begin{equation}\label{eq4}
		\int_{g.(\sigma^N)}\omega_W = \int_{\sigma^N}\omega_W^g,
	\end{equation}
	again using the fact that the integrals of $\omega_W$ along $[o,g.o]$ and $\overline{[o,g.o]}$ cancel each other.
	
	Putting all these equalities together, we obtain
	\begin{equation}\label{eq5}
		\int_{\sigma^N} \omega_W  = \int_{\sigma^N}\omega_W^g.
	\end{equation}
	
	Since \[\int_{\sigma^N} \omega_W=N\int_{\sigma} \omega_W,\] we conclude that
	\begin{equation}\label{eq-eq1}
		\int_{\sigma} \omega_W  = \int_{\sigma}\omega_W^g
	\end{equation}
	for any closed loop $\sigma$ in $W$ based at $o$. The forms $\omega_W$ and $\omega_W^g$ represent
	well-defined elements of $H^1(W, \bR)$, by their very definition. By Equation
	\ref{eq-eq1} above they have the same periods, and since they are both closed differential forms, they are cohomologous. 
\end{proof}

The cohomological consequence of preserving homological lines with holes in quotients is the following (cf.\ Remark \ref{rmk-caveat}):

\begin{lem}\label{lem:Q-form}
	Let $Q=\Ga/K$, let $g\in\comme_G(\Ga)$ preserve homological lines with holes in $Q$, and
	let $\omega\in H^1(Q,\bR)$. Then the periods of $[\omega_W]$ and $[\omega_W^g]$ agree, 
	where $W$ is the common minimal cover of  $S=X/\Ga$ and  its conjugate 
	manifold  $S^g=X/\Ga^g$, and where $\omega_W$ is the pullback of $\omega$ to  $H^1(W,\bR)$.
\end{lem}

\begin{proof}
	Let $\omega\in H^1(Q,\bR)$ be a nontrivial cohomology class. Then the quotient map $q\colon \Ga\longrightarrow Q$ induces a pullback form
	$q^*\omega\in H^1(\gam,\bR)$, which can be  viewed as  a differential form on $S=X/\gam$.
	The map $q$ also induces a map $q_*\colon H_1(\Ga,\bZ)\longrightarrow H_1(Q,\bZ)$.
	If $\sigma$ is any $1$--cycle on $X/\Ga$ then by definition \[\int_{\sigma} q^*\omega=
	\omega(q_*\sigma),\] where the right hand side denotes the  evaluation of $\omega$ on  $q_*(\sigma)$ (recall $\omega$ is a cohomology
	class of $Q$).
	Writing $\omega_W$ for
	the form on $W$ given by pullback of $q^*\omega$ along the covering map $p\colon W\longrightarrow S$,
	we have that $\omega_W^g$ and $\omega_W$ have the  same periods, provided that $g$
	preserves homological lines with holes in $Q$. A justification of this claim is
	identical to that in the proof of Lemma~\ref{trvialcohomolaction}.
\end{proof}

We note the following easy observation (cf.~Observation~\ref{obs:monoid} above).

\begin{obs}\label{obs:monoid-hom}
	Consider the set $C\subset\comme_G(\gam)$ consisting of elements which preserve homological lines with holes in
	$Q$. Then $C$ is closed under
	multiplication of group  elements and contains the identity, and is therefore a monoid. In  particular, if $K_1,K_2\subset C$ are subgroups,
	then the group  $\langle K_1,K_2\rangle$ is contained in $C$.
\end{obs}

\subsection{Hodge theory}\label{sec:patterns}
Hodge theory will allow us to leverage preservation of homological lines with holes in order
to promote equality of cohomology
classes to  equality of forms.
%In this subsection we recall some essential tools on harmonic forms and discrete patterns from existing literature.
We recall the necessary tools from Hodge theory and $L^2-$cohomology that we shall need. Let $M$ be a 
(not necessarily compact)  Riemannian manifold.
We fix notation: $\Omega^k$ will denote the space of smooth $k-$forms, $d$ will denote the differential on forms,
$\ast$ will denote the Hodge star operator, $d^\ast$ will denote the adjoint of $d$, and 
$\Delta=dd^\ast + d^\ast d$ will denote the Laplacian on forms. A form $\omega\in \Omega^k$ is a  \emph{ harmonic $k-$form} for the given
metric on $M$ if $\Delta \omega = 0$. Harmonic forms are closed and co-closed.

\begin{theorem}\cite[Ch. 6]{warner-book}\label{theorem-hodge}
	Let $M$ be a compact Riemannian manifold. Then for all $k$ and every real cohomology class $[\omega] \in H^k(M,\Re)$,
	there exists a unique harmonic form $\omega_{\harm}$ representing $[\omega]$.
\end{theorem}

We shall need a version of Theorem \ref{theorem-hodge} for non-compact complete manifolds $M$. The appropriate cohomology 
theory used is $L^2-$cohomology. Let $\Omega^k_2$  denote the space of smooth square-integrable $k-$forms. The reduced 
$L^2-$cohomology groups are given by \[H^k_{(2)} (M) = \left. ker(d)\middle/\overline {Im(d)}\right.,\] where $\overline {Im(d)}$ 
denotes the closure of the image of
$d$. We refer the reader to \cite{carron} for more details. We shall need only the following special case 
(see  \cite[Lemma 1.5]{carron} due to Gaffney, or  \cite{cheeger} for instance):

\begin{theorem}\label{theorem-hodgel2}
	Let $M$ be a complete negatively curved manifold of finite volume modeled on $\bH^n$ or $\bC \bH^n$.
	Then for  every real cohomology class $[\omega] \in H^1_{(2)}(M, \Re)$, there exists a unique $L^2$
	harmonic form $\omega_{\harm}$ representing  $[\omega]$.
\end{theorem}

\begin{comment}
	content...

Besides the Hodge theorem, our interest  in $L^2$--cohomology  arises from its relationship to cohomology with compact supports. 
Let  $H^p_c(\ast)$ denote  cohomology with compact supports. The following is a simplified statement of Lemma 1.93
of~\cite{lueckbook}:

\begin{lem}\label{lem:lueck}
	Let $M$ be a complete Riemannian manifold  without boundary, and let \[i\colon  H_c^p(M,\bR)\longrightarrow H^p(M,\bR)\]  be the natural
	map. Then the image  of $i$  injects into the  reduced  $L^2$--cohomology $H^p_{(2)}(M,\bR)$.
\end{lem}

In Lemma~\ref{lem:lueck},
\end{comment} 

Note that a compactly supported cohomology class is an $L^2$ class.
Thus in  our  context, if  $X/\Ga$ has nontrivial real cohomology with compact supports then we can find nontrivial $L^2$ harmonic
forms representing  such cohomology classes. In our analysis of the case $b_1(Q)>0$ for groups  arising as quotients of
non-uniform lattices $\Ga$, the absence of a nonzero $L^2$ harmonic $1$--form
will (roughly) allow us to assume  that $H_c^1(X/\Ga,\bR)=0$.
See the  proof of  Theorem~\ref{thm:q-b1} below.

%, which will feed  into the discrete pattern machinery described in this section.

We recall the setup of Lemma~\ref{trvialcohomolaction}
in a slightly restricted setting: we are  given a lattice $\Ga$ in $G\in\{\so(n,1),\su(n,1)\}_{n\geq 2}$ with
associated symmetric space of noncompact type $X$, and an element  $g\in G$ commensurating $\Ga$. We have an orbifold $S=X/\gam$,
 the conjugate manifold $S^g=X/\Ga^g$, the common refinement $W=X/(\Ga\cap\Ga^g)$ and a cohomology class  $\omega\in H^1(S,\bR)$.
We assume the existence of a (possibly $L^2$) harmonic representative $\omega_{\harm}$ of $\omega$,
whose uniqueness is then guaranteed by Theorems~\ref{theorem-hodge}
 and~\ref{theorem-hodgel2}. Note that such a harmonic representative may not exist only in the case where $S$ is noncompact.
 
 We will also call the resulting harmonic form $\omega$  as it will not cause confusion.
 Recall the  notation \[p:X \to S,\quad W,\quad \omega_W,\quad  \omega_W^g\] from Section \ref{ss:pa2}.
For convenience, we will denote $p^*\omega$ by $\omega_X$ and $g^*\omega_X$ by  
$\omega_X^g$, where $g^*$ is the action on 1-forms induced by 
 the  isometry $g$ of $X$.
 
 \begin{cor}\label{cor:hodge-paction}
Assume the above setup, and
 suppose that $g$ preserves homological lines with holes in $\gam$. Then the harmonic representatives of $\omega_W$ and
 $\omega_W^g$ are equal as differential
 1-forms on $W$. In particular, the harmonic representatives of
 $\omega_X$ and $\omega_X^g$ are equal.
\end{cor}
\begin{proof}
Since $g$ acts on $X$ by  an isometry, the pullback of a harmonic  form  under $g$ is  also harmonic
(see Section 4 of~\cite{EL78}, for example).
Thus, $\omega_W^g$ is a
form on  $W$ which is  cohomologous to
the form $\omega_W$, by Lemma~\ref{trvialcohomolaction}.
Since $\Ga\cap\Ga^g$ has finite index in $\Ga$, we have that $W$  still has finite
volume and hence the suitable Hodge  theorem (Theorem \ref{theorem-hodge}  or \ref{theorem-hodgel2}) applies, whence the harmonic
representatives of $\omega_W$  and $\omega_W^g$ are equal. The equality
of forms on $X$  is immediate.
\end{proof}

A  part of the remainder of the paper will deal with the case where there is no  harmonic form representing a nontrivial homology
class, which is to say a complement to Corollary~\ref{cor:hodge-paction} adapted to cusped orbifolds.

\subsection{The commensurator of a form}

The notion of the commensurator of a form will now be introduced. It will be shown that under suitable hypotheses,
$K^G$ lies in the commensurator
of a harmonic form, as is forced by preservation of homological lines with forms. The rigid nature
of the harmonic form will force it to be zero whenever $K^G$ fails to be discrete, which only occurs if $\comme_G(K)$ is dense.
As  before, cohomology with compact supports will be denoted by $H^\ast_c(\, \cdot \,)$.

\begin{defn}\label{def-formpres} Let $\Ga < G$ be a lattice in a semi-simple Lie group $G$ with associated symmetric space $X$, and let $S = X/\Ga$. Let $\omega$ be a closed form such that
	$[\omega] \in H^p(S,\bQ)$ or $[\omega] \in H^p_c(S,\bQ)$ is a non-zero cohomology class. 
	Let $p : X \longrightarrow S$ denote the universal cover.
	The \emph{commensurator} $\comme(\omega)$ of the form $\omega$ is defined as 
	\[\comme(\omega) = \{ h \in G \mid h^\ast p^\ast \omega = p^\ast \omega \}.\]
	A subgroup $H$ of $G$ is said to commensurate $\omega$ if $H < 
	\comme(\omega).$ It is immediate the $\comme(\omega)$ is itself a group.
\end{defn}

We have the  following general  discreteness result that applies to the isometries of real and complex hyperbolic spaces.
We will not  consider isometries of quaternionic hyperbolic spaces or the Cayley plane (see the remarks following
Theorem~\ref{thm:main-general}). We direct the reader to
~\cite{venky,agol-form}, from which the main idea used in the following Proposition is taken:
\begin{prop}\label{prop-generalG}
	Let $X$ be $\bH^n$ or $\bC \bH^n$. For $\Ga$  a torsion-free lattice, let  $S = X/\Ga$. Let $\omega$ be
	a non-zero harmonic or $L^2-$harmonic 1-form according to whether $S$ is compact or non-compact.  Then
	$\comme(\omega)$ is discrete.
\end{prop}
\begin{proof} Let $p: X \longrightarrow S$ denote the universal cover.
	We now argue by contradiction. Suppose that $\comme(\omega)$ is not discrete. 
	Since the associated  Lie group $G$ (i.e. $SO(n,1)$ or $SU(n,1)$) is simple, it follows that $\comme(\omega)$ is dense in $G$, as
	$\comme(\omega)$ contains the Zariski dense subgroup $\Gamma$.
	Also, since $\comme(\omega)$ preserves $p^\ast (\omega)$, we have that $G$ must preserve $p^\ast (\omega)$, since $G$ is
	identified with the group of isometries of $X$. That is, $p^\ast (\omega)$ is a $G-$invariant non-zero 
	harmonic 1-form on
	$X$. (Note that here, compactness or non-compactness of $S$
	is not relevant, as $p^\ast (\omega)$ being defined on $X$ is all that we
	are concerned with at this stage).
	Hence $p^\ast (\omega)$ gives 
	a non-zero harmonic differential 1-form $\omega^\ast$ on the compact dual of $\bH^n$ or $\bC \bH^n$~\cite{venky,agol-form} (cf.~
	Sections 2 and 3 of Chapter II in~\cite{bo-wa}).
	%Precisely, we write $\mathfrak{g}=\mathfrak{t}\oplus\mathfrak{p}$ for the Cartan decomposition of the complexified Lie algebra
	%$\mathfrak{g}$ of $G$, with $\mathfrak{t}$ the  complexified  Lie algebra of the maximal compact subgroup
	%$\mathcal{K}$ of $G$. Then  the Matsushima--Kuga formula
	%(c.f. Corollary 3.4 of Chapter VII  of~\cite{bo-wa} and~\cite{matsushima}) implies that
	%$G$--invariant harmonic differential $1$--forms on $X$  are identified with $\mathrm{Hom}_{\mathcal{K}}(\mathfrak{p},\bC)$, which in
	%turn is  identified with the first cohomology of the compact  dual of $G$ with complex coefficients.
	Since the compact duals  $S^n$ and  $\bC\mathbb{P}^n$ of  $\bH^n$ and $\bC \bH^n$  respectively have trivial
	first cohomology (at least when $n\geq 2$), this is a contradiction.
\end{proof}

From Lemma~\ref{lem:Q-form}, we obtain the following consequence:

\begin{cor}\label{cor:Q-form}
Suppose $\Ga$ is torsion-free.
Let $Q=\Ga/K$,   and let $C\subset\comme_G(\gam)$ denote the set of elements which
preserve homological lines with holes in $Q$.
If there exists a (possibly $L^2)$ harmonic form on $S=X/\Ga$ representing a  pullback  of a nonzero cohomology class
of  $Q$, then $C$ is discrete.
\end{cor}
\begin{proof}
Let $\omega$ be the harmonic representative of  a form on  $S$ arising by pullback from $Q$, and let $g\in C$. Then 
by Lemma~\ref{lem:Q-form} and Corollary~\ref{cor:hodge-paction},
we have that $\omega_W=\omega_W^g$ as forms, by either classical or $L^2$--Hodge theory, and where 
here $W$ is  the common refinement of $S$ and its
conjugate $S^g$. Pulling back  these forms to the universal cover  $X$, we  have  that $g\in\comme(\omega)$. By
Proposition~\ref{prop-generalG}, we conclude that $C$ is discrete.
\end{proof}

\section{Abelian quotients and harmonic 1-forms}\label{sec-mainthm}
We are now in a position to assemble the pieces to prove Theorem~\ref{thm:main-general}.
The ideas to establish  the result naturally bifurcate:
\begin{enumerate}
\item The vanishing cuspidal case, amenable to $L^2-$cohomology techniques. For $\pslr$,
this is the case where the underlying hyperbolic surface has genus greater than zero. This part of the argument uses Hodge theory.
\item The non-vanishing cuspidal case, where discrete patterns of horoballs are used to obtain discreteness of the commensurator (See Appendix \ref{sec-app}).
For $\pslr$, this is the case where the underlying hyperbolic surface has genus equal to zero, and compactly supported
cohomology vanishes. This
part of the argument borrows heavily from  the ideas in~\cite{KM18}.
\end{enumerate}

\subsection{Proof of Theorem~\ref{thm:main-general}}

We now establish part of the main result of this paper:

\begin{theorem}\label{thm:q-b1}
Let $\Ga<G$ be a lattice in a rank one simple Lie group.
Let $K<\Ga$ be an infinite normal subgroup, and let $Q=\Ga/K$. If the first Betti number of $Q$ satisfies
$b_1(Q)>0$ then $\comme_G(K)$ is discrete.
\end{theorem}

Here, the lattice may or may not be torsion--free, and may or may not be uniform.
As remarked  in the introduction, we only consider lattices in $\so(n,1)$  and  $\su(n,1)$.
\begin{proof}[Proof of Theorem~\ref{thm:q-b1}]
We begin by passing to a torsion-free finite index subgroup $\Ga'$ of $\Ga$, and by replacing $K$ with the corresponding finite index subgroup
of $K$ given by the corresponding intersection $K\cap\Ga'$.
The resulting subgroup of $K$ is commensurable with $K$ and hence has the same commensurator in $G$ as $K$.
Moreover, by restricting the
quotient map $\Ga\longrightarrow Q$ to $\gam'$, we get a finite index subgroup $Q'<Q$ which also has positive first Betti number. 
Thus without loss of generality,
we will assume that $\Ga=\Ga'$.

Recall that we write
\[K^G=\langle K^g  \mid g\in\comme_G(K)\rangle\] for the subgroup generated by the collection $\{K^g\}$, as 
$g$ ranges over $\comme_G(K)$.
By Proposition~\ref{commens=commesn}, we have that $\comme_G(K)<\comme_G(\gam)$.
By Theorem~\ref{thm:comm-nonab-b1} and Observation~\ref{obs:monoid-hom}, we have that if $y\in K^G$ then
$y$ preserves homological lines with holes in $Q$.

By hypothesis, we have
$H^1(Q,\bR)\neq 0$. Writing $S=X/\Ga$ as usual, we have that $H^1(S,\bR)\neq 0$ since $Q$ is a quotient of $\Ga$ and since
$\gam=\pi_1(S)$. We have that
$S$ is metrically complete and is either compact or noncompact, which yields two possible cases concerning cohomology:
\begin{enumerate}
\item
$S$ is compact. By Theorem \ref{theorem-hodge}, there is a harmonic form $\omega$ on $S$ which represents the pullback of a nontrivial cohomology class of $Q$.
\item
$S$ is not compact. This case bifurcates into further possibilities:
\begin{enumerate}
\item
The composition \[H^1(Q,\bR)\longrightarrow H^1(S,\bR)\longrightarrow H^1(\partial S,\bR)\] has a nontrivial kernel,
where the first map is the pullback along  the quotient map
$\Ga\longrightarrow Q$ and the second map is the pullback along the inclusion map $\partial S\longrightarrow S$. 
Note  that the first arrow is an injection. Furthermore, \[H^1((S, \partial S),\,\bR) =
H^1_c (S,\, \bR) = H^1_{(2)}(S,\, \bR);\] see~\cite[Lemma 1.93]{lueckbook}.
Hence, by Theorem \ref{theorem-hodgel2}, there is a nonzero cohomology class of $S$
that is represented by
a nonzero $L^2$ harmonic form $\omega$ such that $[\omega]\in H^1_{(2)}(S,\bR)$ is the pullback of a cohomology class of $Q$.
\item
The composition \[H^1(Q,\bR)\longrightarrow H^1(S,\bR)\longrightarrow H^1(\partial S,\bR)\]
is injective.
\end{enumerate}
\end{enumerate}
In case (2), we interpret $\partial S$ in the usual way, i.e.~by removing a small horoball around the cusps of $S$, whereby the boundary
of $S$ becomes the image of the horosphere bounding the horoball.

 Suppose first that there exists a nontrivial (possibly $L^2$) harmonic form on
$S=X/\Ga$ representing a pullback of a nontrivial class
in $H^1(Q,\bQ)$, as in case (1) or (2)(a) above. Then $K^G$ is discrete by Corollary
~\ref{cor:Q-form}. That $\comme_G(K)$ is discrete now follows from Lemma~\ref{lem:comm-discrete0}.

If no such form exists, then
 we are in case (2)(b). Writing $q\colon \Ga\longrightarrow Q$ for the quotient map,
we have that \[q_*\circ i_* \colon H_1(\partial S,\bQ)\longrightarrow  H_1(Q,\bQ)\]
is surjective (where $i: \partial S \longrightarrow S$ denotes  inclusion).
Because $H_1(Q,\bQ)\neq 0$ by hypothesis, there exists a finite collection of cusps $\{T_1,\ldots,T_k\}$ of $S$ which contain
homology classes $z_j \in H_1(T_j,\bQ)$ for which 
$$q_* \circ i_* (z_j) \neq 0.$$
For $1\leq j\leq k$, let $t_j \in \partial X$ denote the basepoint (at infinity) of a horoball lift of $T_j$ to $X$. Let 
	$\TT_j$ denote the set of the $\Ga-$translates of $t_j$ in $\partial X$. Also, let $\HH_j$ (resp.~$\partial \HH_j$)
	denote the collection of horoballs 
	(resp.~horospheres) in $X$ that are lifts of $T_j$ (resp.~$\partial T_j$).
These are an instance of a \emph{discrete pattern} in the sense of Schwartz \cite{schwartz-pihes}, see Definition~\ref{defpatternhoro} below,
 for instance.
	Let $\Ga_j < G$ denote the subgroup preserving the collection $\partial \HH_j$.  
	By \cite[Propositions 5.3 and 5.4]{mahan-pattern} (see Lemma \ref{lem-discretepathoro} for instance),
	the group $\Ga_j$ is a lattice containing $\Ga$ as a subgroup of finite index.

\begin{claim}\label{claim1}
	There is an $N>0$ such that for all 
	%$g\in\comme_G(K)$ arbitrary and 
	$y\in K^G$, we have \[y^N\in \bigcap_{s=1}^k\Gamma_s.\]
\end{claim} 

We complete the proof assuming  Claim~\ref{claim1}.
It follows from
 Claim~\ref{claim1}  that each element of $K^G$ has a uniformly bounded power  contained in the discrete group $\cap_{s=1}^k\Gamma_s$. 
Hence $K^G $ is discrete by Lemma~\ref{lem:power}. Lemma~\ref{lem:comm-discrete0}
now implies that
$\comme_G(K)$ itself is discrete. 
\end{proof}

\begin{proof}[Proof of Claim \ref{claim1}:]
By  Theorem~\ref{thm:comm-nonab-b1}, we know that
	$K^G$ preserves homological lines with holes in $Q$. Choose parabolic subgroups
	$\{G_1,\ldots,G_k\}$ of $G$, which we use to identify $\pi_1(T_j)$ as a subgroup of $\pi_1(S)$ for $1\leq j\leq k$,
	 and let $\{x_1,\ldots,x_k\}\subset\partial X$ be their respective fixed points. 
	Let $\gamma\in \Ga$ be a parabolic isometry representing $z_j \in H_1(T_j,\bZ)$,
	and such that $q_\ast\circ i_\ast (z_j)$ is non-zero. 
	Replacing $\gamma$ by a conjugate in $\Ga$ if necessary, $\gamma$ fixes $x_j$ and hence lies in $G_j$. 
	Let $y\in K^G$.
	Since $y$ preserves homological lines with holes in $Q$,
	 there exists a positive integer $m$ such that \[[(\gamma^m)^y] = [\gamma^m] =m\cdot q_\ast\circ i_\ast (z_j),\]
	 where $[\, \cdot\,]$ denotes the corresponding
	homology class in $H_1(Q,\bZ)$, and where elements of $\gam$ acquire homology classes in $H_1(Q,\bZ)$
	via $q_\ast$. Since $y\in G$, we have that $(\gamma^m)^y$ is also parabolic.
	 Since $y$ commensurates $\Ga$ (by Proposition~\ref{commens=commesn})
	 and preserves homological lines with holes in $Q$,
	 we have that there exists $r \in \Ga$ such that $(\gamma^m)^{yr}\in G_{\ell}$ for some
	 $1\leq\ell\leq k$.
	 Thus, $y$ preserves homological lines with holes in $Q$ but may ``change
	 the cusp" which supports a given cuspidal homology class.
	 Since there are only $k$ many cusps of $S$ which contribute to the homology of $Q$ via
	 $q_\ast\circ i_\ast$, we may assume that $(\gamma^m)^{y^N}$ is conjugate into
	 $G_{j}$ by an element $r\in\Ga$, for $N=k!$.
	We thus have that $y^Nr \in G_{j}$.
	  
	  Now, any element of the parabolic subgroup $G_j$ can be decomposed as $A_\lambda N_\lambda$,
	  where $A_\lambda$ acts on $\partial X \setminus \{x\}$ by a conformal homothety and $N_\lambda $ acts by an isometry. Here,
	  	the metric on   $\partial X \setminus \{x\}$ is obtained by identifying it with a reference horosphere in $X$ based at $x$ via  projection along geodesics from $x$.
		
	  For $X=\bH^n$, these are all Euclidean similarities and for $X=\bC\bH^n$, these are all Heisenberg similarities
	  (see \cite[Section 8.1]{schwartz-pihes}). In particular,
	for any $j$, and for any $g \in G_j$, $g$ scales all distances on the reference horosphere by 
 a fixed $r_g >0$. We call $r_g$ the \emph{scale factor} of $g$. Let \[\widehat{g}: H_1(T_j) \longrightarrow H_1(T_j)\]
 denote the induced map on $H_1(T_j)$ \emph{thought of as a
 	subset of $\partial X \setminus \{x\}$.} Here, we use the notation $\widehat{g}$ in place of $g_*$ to avoid confusing with the action
 on homology of the cusp \emph{per se}.
 Since $g$ scales the length of all
elements by $r_g$, it follows that $\widehat{g} (u) = r_g\cdot u$
for all $u \in H_1(T_j)$.
	  Let $A_\lambda (y^Nr) >0$ denote the scale factor of the homothety component of
	  $y^Nr$. Write $\HH_{x_j} \in \HH_j$ for the horoball in $X$ based at $x_j$.
	  
	  Since 
	\[[(\gamma^m)^{y^Nr}] = A_\lambda (y^Mr)  [\gamma^m]\in H_1(Q,\bQ),\]
	the scale factor $A_\lambda (y^Nr)$ must equal one. But $A_\lambda (y^Nr) =1$
	if and only if $y^Nr$ preserves the horosphere $\partial \HH_{x_j}$. Since $r \in \Ga$ necessarily preserves $\partial \HH_j$, it follows
	that $y^N$ stabilizes 
	$\partial \HH_j$, i.e.\ $y^N \in \Ga_j$. Since $y\in K^G$ and $1\leq j\leq k$ were arbitrary, and since
	$\gam_j$ contains $\cap_{s=1}^k\gam_s$ with finite index (as follows easily from
	 Lemma~\ref{lem-discretepathoro}) this completes the proof of the claim.
\end{proof}

\subsection{Applications}
We conclude this section by giving three sets of examples to which Theorem \ref{thm:q-b1} applies.\\

\noindent {\bf Irrational pencils in complex hyperbolic manifolds:} Many cocompact arithmetic lattices in $\mathrm{SU}(2,1)$
admit irrational pencils, i.e.\
$S= X/\Ga$ admits a holomorphic fibration (with  singular fibers) onto a Riemann surface of genus greater than zero.
Let $F$ denote the general fiber and $i: F \longrightarrow S$ denote inclusion. Then $K=i_\ast (\pi_1(F))$ is normal in $\Ga$ and
$Q=\Ga/K$ is a surface group. Theorem \ref{thm:q-b1} applies to show that $\comme_G(K)$ is discrete.
We note that M.\ Kapovich in unpublished work~\cite{kapovich-morse} (see \cite{bmp} for a small generalization) established 
that $K$ is never finitely presented. \\

\noindent {\bf Real hyperbolic manifolds that algebraically fiber:} Agol \cite{agol-vhak} shows
 that hyperbolic 3-manifolds virtually fiber over the circle with surface group fibers. 
 The resulting normal surface subgroups were dealt with in \cite{llr} without the arithmeticity hypothesis. 
 However, a new family of examples of finitely generated (but
not necessarily 
finitely presented) normal subgroups of arithmetic hyperbolic
$n-$manifolds has recently been discovered. A classical result of Dodziuk \cite{dodziuk,anghel} shows 
that the first $L^2-$betti number of a hyperbolic manifold of dimension greater than 2 vanishes. Kielak \cite{kielak}
shows that a cubulated hyperbolic
group $\Ga$ is virtually algebraically fibered (i.e.\ $\Ga$ admits a virtual surjection to $\bZ$ with a finitely generated kernel) if and only if  
$\beta^1_{(2)} (\Ga) = 0$. On the other hand, Bergeron-Haglund-Wise
\cite{bhw} show that standard cocompact arithmetic congruence subgroups $\Ga$ of $\mathrm{SO}(n,1)$ are cubulated.
Thus standard cocompact arithmetic congruence subgroups $\Ga$ of $\mathrm{SO}(n,1)$ admit finitely generated normal subgroups $K$ with
quotient $\bZ$. This furnishes a family of examples $K$ to which 
Theorem \ref{thm:q-b1} applies  to show that $\comme_G(K)$ is discrete (since $b_1(Q) = b_1(\bZ) = 1$
in this case). 

\noindent {\bf Uncountably many pairwise non-isomorphic $2$--generated groups:} P.~Hall produced uncountably
many pairwise--non-isomorphic quotients of a free group $F_2$ on two generators (see III.C.40 of~\cite{delabook}, for instance).
Evidently, the free group on two generators can be realized as a lattice in a rank one simple Lie group. Hall's construction produces
uncountable families of $2$--generated torsion--free solvable groups, and each of his groups surjects to $\bZ$. This furnishes a continuum's
worth of thin normal subgroups of lattices to which Theorem~\ref{thm:q-b1} applies.

\appendix
\section{Discrete Patterns of horoballs}\label{sec-app}
In the course of the proof of Theorem~\ref{thm:q-b1} (Case 2b), we have used the fact that a certain \emph{discrete pattern of horoballs}
is preserved by $K^g$. Since the notion of a discrete pattern also makes its appearance in earlier approaches to Question~\ref{mainq}, we give a quick account here.

Let $G$ be a rank one  semi-simple Lie group and let $X$ be the associated symmetric
space. The space $X$ is, in a natural way, a Riemannian manifold endowed with a left--invariant metric~\cite{Helgason}.
Following \cite{schwartz-pihes,schwartz-inv,mahan-pattern,biswasmj} we define the following
(see \cite[Definition 1.6]{mahan-pattern} in particular):

\begin{defn}\label{defpatternhoro} Let $\Ga < G$ be a lattice and $S = X/\Ga$.
	A \emph{ $\Ga-$discrete pattern
		of points} on $X$ is a nonempty $\Ga-$invariant set $\SSS \subset X$ such that $\SSS/\Ga$ is finite. 
	
Let $\Ga < G$ be a non-uniform lattice,
and let $S = X/\Ga$. A \emph{ $\Ga-$discrete pattern
	of horoballs} in $X$ is a non-empty $\Ga-$invariant collection $\SSS \subset X$ of closed horoballs such that
$\SSS/\Ga$ is a disjoint union of neighborhoods of cusps. 
\end{defn} 

\begin{defn}\label{def-patpres} Let $\Ga < G$ be a lattice.
	A subgroup $H$ of $G$ is said to \emph{preserve} a $\Ga-$discrete pattern $\SSS$
	 points  if $h(\SSS) \subset \SSS$ for all $h \in H$.
\end{defn}

Propositions 3.5 and 3.7 of \cite{mahan-pattern}  show that a subgroup $H$ of $G$ 
preserving a $\Ga-$discrete pattern $\SSS$ is closed and totally
disconnected. Since any such subgroup of $G$ is necessarily discrete, we have the following:
\begin{lem}\label{lem-discretepat}\cite[Propositions 3.5 and 3.7]{mahan-pattern}
	Let $\Ga < G$ be a lattice and $\SSS$ a $\Ga-$discrete pattern 
	(of points or geodesics). Then the  subgroup $H$ of $G$  preserving  $\SSS$ is  discrete, and $[H:\Ga]<\infty$.
\end{lem}

Propositions 5.3 and 5.4 of \cite{mahan-pattern} (see also \cite[Theorem 3.11]{mahan-relrig}) prove that the  subgroup $H$
of $G$ preserving a $\Ga-$discrete pattern
of horoballs is closed and totally disconnected. It follows that:

\begin{lem}\label{lem-discretepathoro}\cite[Propositions 5.3 and 5.4]{mahan-pattern}
	Let $\Ga < G$ be a non-uniform lattice in a rank one Lie group and $S = X/\Ga$,
	where $X$ is the associated symmetric space. Let $\SSS$ be a $\Ga-$discrete pattern of horoballs.
	Then the  subgroup $H$ of $G$  preserving  $\SSS$ is  discrete, and $[H:\Ga]<\infty$. 
\end{lem}

As an aside, we mention that 
for lattices in $\pslr=\so(2,1)$,
there are more  direct ways of understanding discrete patterns, and in particular Proposition~\ref{prop-generalG} above,
that are inspired by ideas from
Teichm\"uller theory. In this context, one can view the commensurator of a nontrivial harmonic form as explicitly producing a $\Ga$--discrete
pattern. Specifically, one can use the fact that a harmonic form is the real part of an abelian differential on the Riemann 
surface $\bH^2/\Ga$. In the case of a cocompact lattice,
one can use the fact that the set of zeros of the form is nonempty and discrete, and preserved by the commensurator of  the form. Then 
Lemma \ref{lem-discretepat} gives discreteness of the commensurator itself. In the case of a nonuniform  lattice, one uses saddle connections
in the Baily--Borel--Satake compactification of $\bH^2/\gam$,  and the fact that these are invariant under  the commensurator. Again, Lemma
\ref{lem-discretepat} gives discreteness of the commensurator.

\section*{Acknowledgments} The authors are grateful to D. Witte Morris, T. N. Venkataramana, S. Varma and C. S. Rajan for
 helpful discussions and correspondence, in particular for telling us about Proposition \ref{commens=commesn}. 
 Special thanks are due
to D. Fisher and W. van Limbeek for many discussions and for catching an error in an earlier version of this paper.
 We thank D. Wise for 
his lectures on coherence at IMPAN, Warsaw, and for  references on $L^2-$Betti numbers \cite{dodziuk}.
We thank Y. Shalom for his interest and perceptive comments.

%%%%%%%%%%%%%%%%%%%%%%%%%%%
% END of body
%%%%%%%%%%%%%%%%%%%%%%%%%%%

\bibliographystyle{amsplain}

\printindex
\end{document}